\documentclass[a4, 10pt]{amsart}
\usepackage{amssymb,amstext,amsmath,amscd,amsthm,amsfonts,enumerate,graphicx,latexsym,color}
\usepackage{amsmath}
\usepackage{amssymb}
\usepackage{amsthm}
\usepackage{amscd}
\usepackage{mathrsfs}
\usepackage[all]{xy}
\DeclareMathOperator{\Hom}{Hom}
\DeclareMathOperator{\Ext}{Ext}

\DeclareMathOperator{\Specc}{Spec}
\DeclareMathOperator{\Ass}{Ass}

\DeclareMathOperator{\Ker}{Ker}
\DeclareMathOperator{\Coker}{Coker}
\DeclareMathOperator{\Ima}{Im}

\DeclareMathOperator{\depth}{depth}

\DeclareMathOperator{\E}{E}

\DeclareMathOperator{\Tr}{Tr}

\newcommand*{\dpp}{d_{p}}

\newcommand*{\Spec}{S}
\newcommand*{\TTT}{T^{U}}
\newcommand*{\TTTT}{T^{ X_{t(\lambda_{W})+1}(C)}}
\newcommand*{\TTTTT}{T^{X_{t(\lambda_{W})+i-1}(C)}}
\newcommand*{\SSSS}{\widetilde{S}^{f}}
\newcommand*{\SSS}{\widetilde{S}^{C}}

\newcommand*{\sg}{\text{$t(\lambda_{W})$}}
\newcommand*{\eg}{\text{$\E^{i+1}(\lambda_{W})$}}

\newcommand*{\rgt}{\text{$\mathrm{r.grade}_{R}(M,C)$}}
\newcommand*{\ag}{\text{$\mathcal{A}_{C}$}}
\newcommand*{\oo}{\emptyset}
\newcommand*{\OO}{O}

\def\gc{\operatorname{G}_C}

\newtheorem{theorem}{Theorem}[section]

\theoremstyle{definition}

\theoremstyle{remark}

\theoremstyle{corollary}
\newtheorem{corollary}[theorem]{Corollary}
\theoremstyle{proposition}
\newtheorem{proposition}[theorem]{Proposition}
\theoremstyle{definition}

\numberwithin{equation}{section}

\begin{document}
\title{On the reduced grades of modules over commutative rings}
\author{Yoshinao Tsuchiya}
\address{Graduate School of Mathmatics, Nagoya University, Furocho, Chikusaku, Nagoya, Aichi 464-8602, Japan}
\email{m15034w@math.nagoya-u.ac.jp}
\begin{abstract}
Let $R$ be a commutative Noetherian ring.
Recently, Dibaei and Sadeghi have studied the reduced grade of a horizontally linked $R$-module $M$ of finite $\gc$-dimension, where $C$ is a semidualizing $R$-module.
In this paper, we highly refine their results.
In particular, our main result removes the assumptions that $M$ is holizontally linked and $M$ has finite $\gc$-dimension.
\end{abstract}
\maketitle

\section{Introduction}
A purpose of this paper is to generalize some results of a paper of Dibaei and Sadeghi \cite{S4}.
In \cite{S4} they proved many theorems for a horizontally linked module of finite $\gc$-dimension.
However, the conditions ``horizontally linked'' and ``finite $\gc$-dimension'' are too strong for some theorems.
Let $R$ be a commutative Noetherian ring, $M$ a finitely generated $R$-module, and $C$ a semidualizing module.
We say that $M$ satsfies the condition $\widetilde{S}^{C}_{i}$ if $\depth M_{p}\ge \inf \{i, \depth C_{p} \} $ for integer $i$ and all $p\in \Specc R$.
Recall that $M$ is called $n$-$C$-torsionfree if $\Ext_{R}^{i}(\Tr_{C} M,C)= 0$ for any $1\le i \le n$.
Let $t^{C}(M)$ stand for the supremum of the integers $n$ such that $M$ is $n$-$C$-torsionfree.
The main result of this paper is the following.

\begin{theorem}\label{main}
If $t^{C}(M)< \infty$, then the following are equivalent.
\begin{enumerate}[\rm(1)]
\item
For every integer $i$, $M$ is $n$-$C$-torsionfree if and only if $M$ satsfies $\widetilde{S}^{C}_{i}$.
\item
There exists an associated prime $p$ of $\Ext_{R}^{t^{C}(M)+1}(\Tr_{C}M,C)$ such that $t^{C}(M)+1 \le \depth C_{p}$.
\end{enumerate}
\end{theorem}

Auslander and Bridger proved the following; see \cite[Theorem 4.25]{S1}.
\begin{enumerate}[\rm(a)]
\item
If $M$ is $i$-$R$-torsionfree, then $M$ satsfies $\widetilde{S}^{R}_{i}$ for an integer $i$.
\item
If $M$ has finite G-dimension, then $M$ satisfies $i$-$R$-torsionfree if and only if $M$ satsfies $\widetilde{S}^{R}_{i}$ for every integer $i$.
\end{enumerate}
Here, the converse of (b) is not true in general.
For example, let $(R,m,k)$ be a non-Gorenstein local ring with positive depth.
Let $M$ be the first syzygy of the $R$-module $k$.
Then $M$ is $1$-torsionfree.
By the depth lemma, we have $\depth M=1$.
Therefore, $M$ satisfies $i$-$R$-torsionfree if and only if $M$ satsfies $\widetilde{S}^{R}_{i}$  for every integer $i$.
But $M$ has infinite $G$-dimension.

Let $\widetilde{S}^{C}(M)$ denote the supremum of the integers $i$ such that $M$ satisfies $\widetilde{S}^{C}_{i}$.
Then Theorem \ref{main}(1) nothing but says that $M$ satisfies $t^{C}(M)=\widetilde{S}^{C}(M)$.
By this we can generalize theorems of Dibaei and Sadeghi.

\section{The definitions}
Throughout this paper, let $\Lambda$, $\Gamma$ be rings, and a module means a left (or right) module. Let $M$ be $\Gamma$-module which have a resolution of finitely generated projective $\Gamma$-module, $C$ be $\Lambda$-$\Gamma$-bimodule, and $n$$\ge 1$ be an integer. 
Let $d$ be a function which detemind a nonnegative integer or $\infty$
 for every 
 $\Lambda$-modules. 
  In this paper we need many definitions. 
  
 We consider the following conditions for $d$.
Let $0\rightarrow K\rightarrow L \rightarrow N \rightarrow 0$ be a short exact sequence of $\Lambda$-modules. 
%For every $p\in \Spec R$
\begin{align*}
%&(d0)\quad d_{1}(0)=\infty \\
&(\widetilde{d1})\quad \text{if }d(L) > d(K)\text{, then } d(K)\ge d(N)+1,\\
&(d1)\quad \text{if }d(L) > d(K)\text{, then } d(K)= d(N)+1,\\
%&(\widetilde{d2})\\
&(d2)\quad \text{if }d(L) > d(N)\text{, then } d(K)=d(N)+1, \\
&(d3)\quad \text{if }d(L) < d(N)\text{, then } d(K)\le d(L). \\
\end{align*}
%&(d3)\quad \text{if }d(L) = d(N)\text{, then } d(K)\ge d(L). \\
 We have $(d1)\Rightarrow(\widetilde{d1})$, and $(\widetilde{d1})(d2)\Rightarrow (d1)$ easily and if $d$ satisfies ($\widetilde{d1}$) we have $d(0)\ge \sup\{ d(N) \mid R\text{-module } N \}$ because there exist a exact sequence $0\rightarrow 0 \rightarrow L\rightarrow L \rightarrow 0$ for any $\Lambda$-module $L$. 
 For example let $H^{i}$ be an $i$-th right derived functor of a left exact functor to an aberian category from the category of $\Lambda$-module. For a $\Lambda$-module $N$,
 we put $d(N)=\inf\{i \mid H^{i}(N)\neq 0\}$. Then $d$ satisfies $(d1)(d2)(d3)$ because there exist the long exact sequence for a short exact sequence.\\ 
  The following Propositions follows immediately.
\begin{proposition}Let $i$ be a nonnegative integer and $ 0 \rightarrow K \rightarrow L^{0}\rightarrow L^{1}\cdots \rightarrow L^{i}\rightarrow N\rightarrow 0$ be a exact sequence of $\Lambda$-module.
Then following holds.
\begin{enumerate}[$(1)$]
\item
If $d$ satisfies $(\widetilde{d1})$(resp. $(d1)$) and $d(K)-j< d(L^{j})$ for $0\le j\le i$, then $d(K)\ge d(N)+i+1$ (resp. $d(K)=d(N)+i+1$).
\item
If $d$ satisfies $(d2)$ and $d(N)+i-j<d(L^{j})$ for $0\le j\le i$, then  $d(K)=d(N)+i+1$.
\end{enumerate}
\begin{proof}
 $(1)$;  Let $k$ be a nonnegative integer. We assume that the assetion is true  for $i=k-1$ and prove the case of $i=k$. 
             We have exact sequences $0 \rightarrow K \rightarrow L^{0}\rightarrow L^{1}\cdots \rightarrow L^{k-1}\rightarrow T\rightarrow 0$   
  and $0\rightarrow T \rightarrow L^{k} \rightarrow N\rightarrow 0$. 
 By the assumption, $d(K)=d(T)+k$. Since $d(K)-k<d(L^{k})$, by the later exact sequence follows $d(K)\ge d(N)+k+1$(resp. $d(K)=d(N)+k+1$) by $(\widetilde{d1})$(resp. $(d1)$).\\
$(2)$; Let $k$ be a nonnegative integer. We assume that the assetion is true  for $i=k-1$ and prove the case of $i=k$.  There exist the exact sequence  $0\rightarrow T\rightarrow L^{1}\rightarrow L^{2}\cdots \rightarrow L^{k}\rightarrow N\rightarrow 0$ and $0\rightarrow K \rightarrow L^{0} \rightarrow T \rightarrow 0$.  
By the assumption, $d(T)=d(N)+k$. Since $d(N)+k<d(L^{0})$, we have  $d(K)=d(N)+k+1$ by the later exact sequence and $(d2)$.
\end{proof}
\end{proposition}
We put
$(-)^{\dag} =\Hom_{\Gamma}(-,C)$.
Let $ \cdots \rightarrow P_{n}\rightarrow P_{n-1}\cdots \rightarrow P_{1}\xrightarrow{d_{1}} P_{0}
\rightarrow M
\rightarrow 0$ be the resolution of finitely generated projective modules of $M$.
%We put\\
%\begin{equation*}
%\begin{alignedat}{2}
%&\Tr_{C}M\colon=\Coker(d_{1}^{\dag} \colon F_{0}^{\dag}\rightarrow F_{1}^{\dag}),
%\end{alignedat}
%\end{equation*}
%We recommend to see \cite {S1} for a detail of $\Tr_{R} M$. However we don't need proppaties of it in this paper in particular. \\ 
Let $W$ be a $\Lambda$-module and $\lambda_{W} \colon W\rightarrow M^{\dag}$ be a homomorphism of $\Lambda$-module.
We put\\
\begin{equation*}
\E^{i}(\lambda_{W}) \colon=
\left\{
\begin{alignedat}{2}
&\Ker \lambda_{W} \quad &(i=1)\\
&\Coker \lambda_{W} \quad &(i=2)\\
&\Ext_{R}^{i-2}(M,C) \quad &(i\ge 3)                           \\
\end{alignedat}
\right..
\end{equation*}
$$\sg \colon=\inf \{i \ge 0\mid \eg \neq 0\}.$$\\
We assume $\Lambda=\Gamma=R$ is a commtative Noethrian ring, $M'$ is a finitely generated $R$-module, $C$ is a semidualizing $R$-module, $W=M'$, $M=M'^{\dag}$ and $\lambda_{M'}; M'\rightarrow M'^{\dag \dag}$ is natural map. Then we have $\E^{i}(\lambda_{W})=\Ext_{R}^{i}(\Tr_{C}M',C)$; see \cite[Proposition 4.13]{S3} and see \cite{S4} for a detail of a semidualizing module.\\
For  $0\le i\le 2$ we put
\begin{equation*}
O^{i}(\lambda_{W}) \colon=\left\{
\begin{alignedat}{2}
&\Coker(d_{1}^{\dag} \colon P_{0}^{\dag}\rightarrow P_{1}^{\dag}) &(i=0)\\
&\Ima(d_{1}^{\dag} \colon P_{0}^{\dag}\rightarrow P_{1}^{\dag})  \quad &(i=1)\\
&W &(i=2) \\
\end{alignedat}
\right..
\end{equation*}
Let $S$ be a non empty set, and $d_{()}$ be a function which determind for $p\in S$ a function which detemind  a nonnegative integer or $\infty$
 for every $\Lambda$-modules. 
 Let $f$ be a function which determind a nonnegative integer or $\infty$ for a element of $S$.
 Let  $i$ be an integer and $U$ be a subset of $S$. 
Let $N$ be $\Lambda$-module. We say that $N$ satsfies $ \widetilde{S}^{f}_{i}$ 
  if $d_{p}(N)\ge \inf \{i, f(p)\}$ for an integer $i$ and all $p\in S$.
We put
\begin{equation*}
\begin{aligned}
&X_{i}(N)\colon=\{p\in \Spec \mid i\le d_{p}(N) \},\\
&Y^{f}(N)\colon=\{p\in \Spec \mid d_{p}(N)<f(p)\},\\
&T^{U}(N)\colon=\inf \{d_{p}(N) \mid p\in U \},\\
&\widetilde{S}^{f}(N)\colon=\sup \{i \mid N \text{ satisfies }\widetilde{S}^{f}_{i}\}.\\
\end{aligned}
\end{equation*}
Note that $0\le T^{U}(N)$ and $0 \le \widetilde{S}^{f}(N)$.
Let $i$ be a integer . The following Propositions hold.
\begin{proposition}
\begin{enumerate}[\rm(1)]
\item
$\widetilde{S}^{f}(N)=\inf \{d_{p}(M)\mid p\in Y^{f}(N)\}$ holds.
\item
If $d_{p}$ satisfies $(\widetilde{d1})$ $($resp. $(d3)$, $(d1)(d2)(d3)$$)$
 for any $p\in \Spec $, then the function  $T^{U}$ also satisfies $(\widetilde{d1})$ (resp. $(d3)$, $(d1)(d2)(d3)$).
\item
If $d_{p}$ satisfies $(\widetilde{d1})$ (resp. $(d3)$) for any $p\in \Spec$, then the function  $\widetilde{S}^{f}$ also satisfies $(\widetilde{d1})$ (resp. $(d3)$).
\end{enumerate}
\end{proposition}
\begin{proof}
$(1)$;
If $Y^{f}(N)$$=$$\oo$, $\widetilde{S}^{f}(N)$=$\infty$ follows by definition. Therefore the assertion holds.
If $Y^{f}(N)$$\neq$$\oo$, we put $l$$=$$\min \{d_{p}(N)\mid p\in Y^{f}(N)\}$.
We choose $p\in Y^{f}(N)$ satisfying $ d_{p}(N)$$=$$l$.
Then, for $q\notin Y^{f}(N)$ 
$d_{q}(N)\ge \inf \{l, f(p)\}$ holds clearly. 
By our choice of $p$ for $q\in Y^{f}(N)$, inequality also holds. Therefore $l\le  \widetilde{S}^{f}(N)$.
By the definition of $\widetilde{S}^{f}(N)$, $d_{p}(N)$$\ge$$ \inf \{ \widetilde{S}^{f}(N), f(p)\}$ holds.
Since $p\in Y^{f}(N)$, 
We get $\widetilde{S}^{f}(N)$$\le$$ d_{p}(N)=l$. Therefore $(1)$ holds. \\
$(2)$; $(\widetilde{d1})$; Let $0\rightarrow K\rightarrow L \rightarrow N \rightarrow 0$ be a short exact sequence of $\Lambda$-module. 
If $\TTT(K)<\TTT(L)$, there exist $p \in \Spec$ such that $\TTT(K)=\dpp(K)<\dpp(L)$ and $p\in U$. Since $\dpp$ satisfies $(\widetilde{d1})$, we have $\dpp(K)\ge \dpp(N)+1$.
Therefore $\TTT(N)+1\le \TTT(K)$.
 $(d3)$; If $\TTT(L) <\TTT(N)$, there exist $p \in \Spec$ $\TTT(L)=\dpp(L)<\dpp(N)$, and $p\in U$. Since $\dpp$ satisfies $(d3)$, $\dpp(L)\ge \dpp(K)$ Therefore $\TTT(L)\ge \TTT(K)$.
$(d1)(d2)(d3)$; We may prove only $\TTT$ satisfies $(d2)$ by a note below of definition of $d$.
If we assume $\TTT(N)<\TTT(L)$, there exist $p \in \Spec$ $\TTT(N)=\dpp(N)<\dpp(L)$ and $p\in U$. Since $\dpp$ satisfies $(d2)$, we have $\dpp(K)=\dpp(N)+1$.
Therefore $\TTT(N)+1\ge \TTT(K)$.
  If $\TTT(K)=\TTT(L)$, we have also $\TTT(N)+1= \TTT(K)$ clearly.
  If $\TTT(K)<\TTT(L)$ we have $\TTT(N)+1=\TTT(K)$ because $T^{U}$ satisfies $(\widetilde{d1})$  .  Therefore $(d2)$ holds.\\  %$(\widetilde{\widetilde{d1}})$;
$(3)$; Let $0\rightarrow K\rightarrow L \rightarrow N \rightarrow 0$ be a short exact sequence of $R$-module. 
%$(\widetilde{d1})$; 
$(\widetilde{d1})$; If $\SSSS(K)<\SSSS(L)$, there exist $p\in \Spec$ such that $\SSSS(K)=\dpp(K)<f(p)$ by $(1)$, and $d_{p}(K)<d_{p}(L)$. 
 Since $\dpp$ satisfies $(\widetilde{d1})$, $\dpp(K) \ge \dpp(N)+1$. Since $f(p)>\dpp(K)> \dpp(N)$, by $(1)$ we have $\SSSS(K) \ge \SSSS(N)+1$.
 $(d3)$; We assume $\SSSS(L)<\SSSS(N)$. there exist $p \in \Spec$ such that $\dpp(L)<\dpp(N)$, and $\dpp(L)< f(p)$. Since $\dpp$ satisfies $(d3)$, $\dpp(K)\le \dpp(L)<f(p)$. Therefore $\SSSS(K)\le \SSSS(L)$ by $(1)$.
\end{proof}
  Let $f$ be the function which is $f(p)=d_{p}(C)$ for $p \in S$. We put $\SSS=\SSSS$, $Y^{C}=Y^{f}$. 
The following Proposition give the relationships among $t$, $\TTT$, $\SSS$, $X_{i}$, $Y^{C}$. 
\begin{proposition}
We assume that $d$ satisfies $(\widetilde{d1})(d3)$. Let $N$ be $\Lambda$-module and $C'$ be a direct summand of a finite direct sum of copies of $C$. the following holds.
\begin{enumerate}
\item[$(1)$]
$i\le T^{ X_{i}(C)}(C'), \quad \SSS(C')=\infty$ holds for any integer $i$. 
%and if  $Y^{C}(N)\neq \emptyset$, $T^{Y_{C}(N)}(N)<T^{Y^{C}(N)}(C')$.
\item[$(2)$]
If $d'$ is a function which satisfies $(\widetilde{d1})$ and $\sg+i-1 \le d'(C')$ for every $C'$ a direct summand of a finite direct sum of copies of $C$, then $\sg+i-2 \le d'(\OO^{i}(\lambda_{W}))$ holds for $0\le i\le 2$.
\item[$(3)$]
%$Y^{C}(W)\subseteq X_{t_{C}(M)+1}(C)$ holds.
%Furthermore if $2\le \sg$, 
We have $Y^{C}(\OO^{i}(\lambda_{W}))\subseteq X_{t(\lambda_{W})+i-1}(C)$ for $0\le i\le 2$.
\item[$(4)$]
%$\sg \le \TTTT(W) \le \widetilde{S}^{C}(W)$ holds.
%Furthermore if $2\le \sg$, 
$\sg+i-2 \le \TTTTT(\OO^{i}(\lambda_{W}))
  \le \widetilde{S}^{C}(\OO^{i}(\lambda_{W}))$ for $0\le i \le 2$ holds.
\item[$(5)$]
%$\sg=\TTTT(\OO^{i})$ if and only if $\sg= \widetilde{S}^{C}(W)$.
%Furthermore if $2\le \sg$, 
$\sg+i-2= \TTTTT(\OO^{i}(\lambda_{W}))$ if and only if $\sg+i-2=\widetilde{S}^{C}(O^{i}(\lambda_{W}))$ for $0\le i \le 2$.
\end{enumerate}
\end{proposition}
\begin{proof}
$(1)$;  $i\le  T^{X_{i}(C)}, \quad \SSS(C)=\infty$
% and if $Y^{d,S}_{C}(N)\neq \emptyset$, $T_{Y^{d,S}_{C}(N)}^{d,S}(N)<T_{Y^{d,S}_{C}(N)}^{d,S}(C')$ 
is obvious by the definition. There exist a short exact sequence $0\rightarrow C\rightarrow \oplus^{s}C\rightarrow \oplus^{s-1}C\rightarrow0$ for $s>1$,
and there exist $k>1$ and a $\Lambda$-module $C''$ such that $0\rightarrow C' \rightarrow \oplus^{k}C\rightarrow C''\rightarrow0$ and $0\rightarrow C'' \rightarrow \oplus^{k}C\rightarrow C'\rightarrow0$ is exact. The assertion follows because $T^{X_{i}(C)}$, $\SSS$ 
%$T_{Y^{d,S}_{C}(N)}^{d,S}$ 
satisfies $(\widetilde{d1})(d3)$ by Proposition $2.2$ $(2)(3)$.
\\
$(2)$; The cases of $\sg=0$ and $\sg=1$, $i=0,1$ is clear.
If $\sg=1$, $i=2$, we have 
 %\begin{equation*}
%\begin{aligned}   
$ 0\rightarrow W \rightarrow M^{\dag} \rightarrow E^{2}(\lambda_{W})\rightarrow 0$, and
$ 0\rightarrow M^{\dag}\rightarrow C^{0}\rightarrow C^{1} \rightarrow O^{0}(M) \rightarrow0.$
Where each $C^{i}$ is a direct summand of a finite direct sum of copies of $C$ for $0\le i \le1$. %\end{aligned}
 %\end{equation*}
 Since $d'$ satisfies $(\widetilde{d1})$, 
 we have $d'(M^{\dag})\ge 2$  by (1) and proposition $2.1$ $(1)$.
 Therefore we have $\sg \le d'(W)$ by the later exact sequence.
 We assume that $\sg \ge 2$. Then we have $W \cong M^{\dag}$. 
Since $\E^{i}(\lambda_{W})=0$ for $2\le i \le \sg $,
 we have exact sequence 
%$$
%\begin{CD}
  $ 0\rightarrow M^{\dag} \rightarrow C^{0}\rightarrow C^{1}\rightarrow \cdots \rightarrow C^{\sg-1}\rightarrow L\rightarrow0.$
%\end{CD}
%$$
Where each $C^{i}$ is a direct summand of a finite direct sum of copies of $C$ for $0\le i \le \sg-1$.
Since $d'$ satisfies $(\widetilde{d1})$, $\sg+i-2 \le d'(\OO^{i}(\lambda_{W}))$ for $0\le i\le 2$ follows.\\
%Therefore since $\dpp(W)\ge \inf \{ \sg, \dpp(C) \}$ $($resp. $\dpp(\OO^{i}_{C}\Tr_{C}M)\ge \inf \{ \sg+i-2, \dpp(C) \}$ for $0\le i\le 1$. $)$
 %we have $\sg < \dpp(C)$ for $p\in Y_{C}^{d,S}(W)$ $($resp. $\sg+i-2<\dpp(C)$ for $p\in Y_{C}^{d,S}(\OO^{i}_{C}\Tr_{C} M)$ for $0\le i\le 1$. $)$. 
 %Therefore $Y_{C}^{d,S}(W)\subseteq X_{t_{C}(M)+1}^{d,S}(C)$ $($resp. $Y_{C}^{d,S}(\OO^{i}_{C}\Tr_{C} M)\subseteq X_{t_{C}(M)+i-1}^{d,S}(C)$ for $0\le i\le 1$. $)$ holds.\\
 $(3)$; Let be $0\le i \le 2$. We take $p\in Y^{C}(\OO^{i}(\lambda_{W}))$ Then we have $\sg+i-2 \le \widetilde{S}^{C}(W)\le \dpp(\OO^{i}(\lambda_{W}))<\dpp(C)$ by $(2)$ and Proposition $2.2(1)(3)$.
  Therefore $p\in X_{t(\lambda_{W})+i-1}(C)$.\\
 $(4)$; 
This is clear by $(2)(3)$ and Proposition $2.2(2)(3)$.\\
 $(5)$;  Let be $0\le i \le 2$. $\sg+i-2= \TTTT(\OO^{i}(\lambda_{W}))$, there exist $p\in \Spec$  such that $\sg+i-2=\dpp(\OO^{i}(\lambda_{W}))<\dpp(C)$. Therefore $\sg+i-2=\widetilde{S}^{C}(O^{i}(\lambda_{W}))$
 by $(4)$ and Proposition $2.2$ $(1)$. The converse is clear by $(4)$.
      \end{proof}
 \section{The main result}
 In this section, we assume $d_{p}$ satisfies $(d1)(d2)(d3)$ for any $p\in \Spec$.  \\ 
Let $N$ be $\Lambda$-module. We put
\begin{align*}
% A^{0}(N)\colon=\{p\in \Spec  \mid d(N_{p})=0\} \\
  A_{i}(N)\colon=\{p\in \Spec  \mid d_{p}(N)=i\}. 
   \end{align*}
%\quad \quad \quad \quad \quad \text{                                                                                       }\quad \quad \quad \quad 

The following theorem is the main result of this paper which was implied in the introduction.
%$pR_{p}$ includes weak $C_{p}$-sequence of length $\depth C_{p}$ and for all $p\in \Spec R$.

\begin{theorem}
 We assume $\sg<\infty$, then
\begin{enumerate}[$(A)$]
\item
$A_{0}(\E^{\sg+1}(\lambda_{W}))\cap X_{\sg+1}(C)= A_{\sg}(W)\cap Y^{C}(W)$.
Furthermore if $\sg \ge 2$ we also have $A_{0}(\E^{\sg+1}(\lambda_{W}))\cap X_{\sg+i-1}(C)=A_{\sg+i-2}(\OO^{i}(\lambda_{W}))\cap Y^{C}(\OO^{i}(\lambda_{W}))$ for $0\le i \le 2$.
\item
$A_{0}(\E^{\sg+1}(\lambda_{W}))\cap X_{\sg+1}(C)\neq \emptyset$ if and only if $\sg=\SSS(W)$.
Furthermore if $\sg \ge 2$ we also have $A_{0}(\E^{\sg+1}(\lambda_{W}))\cap X_{\sg+i-1}(C)\neq \emptyset$ if and only if $\sg+i-2=\SSS(\OO^{i}(\lambda_{W}))$ for $0\le i \le 2$.
%When this is the case, $\sg+i-1=\TTTTT (\OO^{i+1}_{C}\Tr_{C}(M))=T_{Y_{C}^{d,S}(\OO^{i}_{C}\Tr_{C}(M))}^{d,S}(\OO^{i+1}_{C}\Tr_{C}(M))$ holds for $0\le i \le 1$.
\end{enumerate}
\end{theorem}

\begin{proof}
 $(A)$; If $\sg=0$ we have an exact sequence $0\rightarrow \E^{\sg+1}(\lambda_{W})\rightarrow W \rightarrow M^{\dag}$. 
We take $p\in A_{0}(\E^{\sg+1}(\lambda_{W}))\cap X_{\sg+1}(C)$. 
Then we have $\dpp(W)=0$ by the exact sequence since $d_{p}$ satisfies $(\widetilde{d1})$.
 Since $p \in X_{\sg+1}(C)$ we have $0<\dpp(C)$. Therefore we have $p\in A_{\sg}(W)\cap Y^{C}(W)$. Conversly we take  
 $p\in A_{\sg}(W)\cap Y^{C}(W)$. 
 %we have $p\in X_{\sg+1}^{d,S}(C)$ because $\sg=\dpp(W)<\dpp(C)$. 
 Since $\dpp$ satisfies $(\widetilde{d1})$, we get $\dpp(M^{\dag})\ge 1$  by an exact sequence $ 0\rightarrow M^{\dag}\rightarrow C^{0} \rightarrow O^{1}(\lambda_{W}) \rightarrow0.$ (We always denote a direct summand of a finite direct sum of copies of $C$ by $C^{i}$ for integer $i$.)
Therefore we have $\dpp(\Ima(\lambda_{W}))\ge 1$. Therefore since $\dpp$ satisfies $(d3)$, we get $\dpp(\E^{\sg+1}(\lambda_{W}))=0$. Therefore $p\in A_{0}(\E^{\sg+1}(\lambda_{W}))\cap X_{\sg+1}(C)$. If $\sg=1$, we have an exact sequence $0 \rightarrow W \rightarrow M^{\dag}\rightarrow \E^{\sg+1}(\lambda_{W})\rightarrow 0$.
 We take $p\in A_{0}(\E^{\sg+1}(\lambda_{W}))\cap X_{\sg+1}(C)$.  Since $\dpp$ satisfies $(\widetilde{d1})$, we get $\dpp(M^{\dag})\ge 2$  by an exact sequence $ 0\rightarrow M^{\dag}\rightarrow C^{0}\rightarrow C^{1} \rightarrow O^{0}(\lambda_{W}) \rightarrow0$ and Proposition $2.1$ $(1)$. Therefore since $\dpp$ satisfies $(d2)$, we have $\dpp(W)=1$. Therefore we have $p\in A_{\sg}(W)\cap Y^{C}(W)$.   Conversly we take $p\in A_{\sg}(W)\cap Y^{C}(W)$. 
 %we have $p\in X_{\sg+1}^{d,S}(C)$ because $\sg=\dpp(W)<\dpp(C)$. 
 Since $\dpp$ satisfies $(\widetilde{d1})$, we have $\dpp(\E^{\sg+1}(\lambda_{W}))=0$. Therefore $p\in A_{0}(\E^{\sg+1}(\lambda_{W}))\cap X_{\sg+1}(C)$.
We assume that $\sg \ge 2$ $0\le i\le2$. Then we have $W \cong M^{\dag}.$ There exist the exact sequences $0\rightarrow M^{\dag} \rightarrow C^{0}\rightarrow C^{1}\rightarrow \cdots \rightarrow C^{\sg-1}\rightarrow L\rightarrow0 $, $0\rightarrow \E^{\sg+1}(\lambda_{W})\rightarrow L\rightarrow N\rightarrow 0$, and $0\rightarrow N\rightarrow C^{\sg} 
\rightarrow U\rightarrow 0$, where $L, N, U$ is $\Lambda$-modules. We take $p\in A_{0}(\E^{\sg+1}(\lambda_{W}))\cap X_{\sg+i-1}(C)$. We get $\dpp(L)=0$ by the second     
exact sequence because $\dpp$ satisfies $(\widetilde{d1})$. We get $p\in A_{\sg+i-2}(\OO^{i}(\lambda_{W}))\cap Y_{C}(\OO^{i}(\lambda_{W}))$ by the first exact sequence and Proposition $2.1$ $(2)$ since $\dpp$ satisfies $(d2)$.
Conversly we take $p\in A_{\sg+i-2}(\OO^{i}(\lambda_{W}))\cap Y^{C}(\OO^{i}(\lambda_{W}))$. Then we get $\dpp(L)=0$ by second exact sequence and Proposition $2.1$ $(1)$ since $\dpp$ satisfies $(d1)$. We have $\dpp(N)\ge 1$  by the third exact sequence because $\dpp$ satisfies $(\widetilde{d1})$. Therefore since $\dpp$ satisfies $(d3)$, we get $p\in A_{0}(\E^{\sg+1}(\lambda_{W}))\cap X_{\sg+i-2}(C)$.\\
$(B)$; We put $S'=\{1\}$ and $d'_{1}=\TTTT$. 
$($resp. if $2 \le t(\lambda_{W})$, $d'_{1}=\TTTTT$ for $0\le i \le2$$)$.
Then we have $A_{0}(\E^{\sg+1}(\lambda_{W}))\cap X_{\sg+1}(C)= A_{\sg}(W)\cap Y^{C}(W)$ $($resp. if $\sg \ge 2$       
$A_{0}(\E^{\sg+1}(\lambda_{W}))\cap X_{\sg+i-1}(C)= A_{\sg+i-2}(\OO^{i}(\lambda_{W}))\cap Y^{C}(\OO^{i}(\lambda_{W}))$  for $0\le i \le 2$$)$
for $S'$ and $d'_{()}$
by Proposition $2.2$ $(2)$ and $(A)$. We have $A_{0}(\E^{\sg+1}(\lambda_{W}))= A_{\sg}(W)$ $($resp.  if $\sg \ge 2$ $A_{0}(\E^{\sg+1}(\lambda_{W}))=A_{\sg+i-2}(\OO^{i}(\lambda_{W}))$ for $0\le i\le 2$ $)$ by Proposition $2.3$ $(1)$. Therefore for $S$ and $d$ we get 
$A_{0}(\E^{\sg+1}(\lambda_{W}))\cap X_{\sg+1}(C)\neq \emptyset$
if and only if $\sg=\TTTT(W)$ $($resp. if $\sg \ge 2$, $A_{0}(\E^{\sg+1}(\lambda_{W}))\cap X_{\sg+i-1}(C) \neq \emptyset$ if and only if $\sg+i-2=\TTTTT(\OO^{i}(\lambda_{W}))$ for $0\le i\le 2$ $)$. Therefore the assertion follows by Proposition $2.3$ $(5)$. 
%The rest of the assertion follows by Proposition $2.2$ $(2)$. 
\end{proof}
 We can prove some results of a paper of Dibaei and Sadeghi (\cite{S2}) by Theorem $3.1$.
   We recommend to see (\cite{S2})  for various notations which was used in their paper and (\cite{S4}) for a detail of semidalizing module and $G_{C}$-dimension.
   We assume that $\Lambda=\Gamma=R$ is a semiperfect commutative ring and $C$ is a semidualizing $R$-module. We put $\Spec=\Specc R$, $d_{p}(N)=\depth_{R_{p}}(N_{p})$ for $R$-module $N$
.  \begin{corollary}\cite[Lemma 4.4]{S2}
Let $M$ be a horizontally-linked $R$-module of finite and positive $G_{C}$-dimension. Set
 $n=\rgt$. If $\lambda M \in \ag$ $($e.g. $\mathrm{pd}_{R}(\lambda M)<\infty$$)$, then
 $$\Ass \Ext_{R}^{n}(M,C)=\{p\in \Specc R\mid G_{C_{p}}\text{-}\dim_{R_{p}}M_{p}\neq 0, \depth_{R_{p}}((\lambda M)_{p})=n=\mathrm{r.grade}_{R_{p}}(M_{p}, C_{p})\}.$$
\end{corollary}

\begin{proof}
 We put $W=M^{\dag}$, and $\lambda_{W}=id_{M^{\dag}}$. Note we have $\rgt=\sg-1$.   
For $p \in \Ass \Ext_{R}^{n}(M,C)$, 
$G_{C_{p}}\text{-}\dim_{R_{p}}M_{p}\neq 0$ 
and $n=\mathrm{r.grade}_{R_{p}}(M_{p}, C_{p})$ is obvious.
 By $\lambda M \in \ag$ we have $\depth_{R_{p}}((\Tr M)_{p})=\depth_{R_{p}}((\Tr_{C}M)_{p})$;
 see the proof of \cite[Lemma $4.4$]{S2}. 
Since $M$ have positive finite $G_{C}$-dimension we get $A_{0}(\E^{\sg+1}(\lambda_{W}))\subseteq X_{\sg}(C)$ (see \cite[Proposition $6.1.7$ (vi), $6.4.2$]{S4}),   
$\Ass \Ext_{R}^{n}(M,C)=\{p\in \Specc R\mid  \depth_{R_{p}}((\OO^{1}(\lambda_{W}))_{p})=n<\depth_{R_{p}}(C_{p})\}$ holds by Theorem $3.1$. Since $\depth_{R_{p}} ((\OO^{1}(\lambda_{W}))_{p})<\depth_{R_{p}} C_{p}$  we have $\depth_{R_{p}}((\Tr M)_{p})+1 =\depth_{R_{p}}((\Tr_{C} M)_{p})+1=\depth_{R_{p}}((\OO^{1}(\lambda_{W}))_{p})$  and as $\depth_{R_{p}}((\Tr M)_{p})<\depth_{R_{p}}R_{p}$ we have  $\depth_{R_{p}}((\Tr M)_{p})+1=\depth_{R_{p}}((\lambda M)_{p})$. Therefore $\depth_{R_{p}}((\OO^{1}(\lambda_{W}))_{p})=\depth_{R_{p}}((\lambda M)_{p})$
Similarly  if  $\depth_{R_{p}}((\lambda M)_{p})<\depth_{R_{p}}(R_{p})$, $\depth_{R_{p}}((\OO^{1}(\lambda_{W})_{p})=\depth_{R_{p}}((\lambda M)_{p})$ holds. Therefore assertion follows.
\end{proof}

\begin{corollary}\cite[Theorem 4.6 (ii)]{S4}
Let $M$ be a stable $R$-module of finite $G_{C}$-dimension and $\lambda M \in \ag$  $(e.g.$ $\mathrm{pd}_{R}(\lambda M)<\infty$$)$.
 For an integer $n>0$, the following statement hold true.
$$\quad \text{If } M \text{is horizontally linked, then } \rgt \text{$\ge$ n if and only if } \lambda M \text{satisfies }  \widetilde{S}_{n}.$$ 
\end{corollary}

\begin{proof}
We put $W=M^{\dag}$, and $\lambda_{W}=id_{M^{\dag}}$. We may assume $0<G_{C}$-$\dim(M)$.
Since $A_{0}(\E^{\sg+1}(\lambda_{W}))\subseteq X_{\sg}(C)$,
we have $\sg-1=\SSS(\OO^{1}(\lambda_{W}))$ by Theorem $3.1$ $(B)$.
There exist $p\in \Spec$ such that $\sg-1=\dpp(\OO^{1}(\lambda_{W}))<\dpp(C)$ by Proposition $2.2$ $(1)$.
Thererfore we get $\sg-2=\dpp(\Tr_{C}M)$. Since $\lambda M \in \ag$, we have $\sg-2=\dpp(\Tr M)$.
Therefore we get $\sg-1=\dpp(\lambda M)<\dpp(R)$. Therefore we have $\widetilde{S}(\lambda M) \le \sg-1$.
Similarly we have also  $\widetilde{S}(\lambda M) \ge \sg-1$.
Therefore the assertion follows.
\end{proof}  

\begin{corollary}\cite[Theorem 4.13]{S2}
Let $M$ be a horizontally linked $R$-module of finite $G_{C}$-dimension and $\lambda M\in \ag$. then
$$\rgt=\inf \{ \depth_{R_{p}}((\lambda M)_{p})\mid p\in NG_{C}(M)
\}$$
\end{corollary}

\begin{proof}
We put $W=M^{\dag}$, and $\lambda_{W}=id_{M^{\dag}}$. We may assume $0<G_{C}$-$\dim(M)$.
Since $A_{0}(\E^{\sg+1}(\lambda_{W}))\subseteq X_{\sg}(C)$,
we have $A_{0}(\E^{\sg+1}(\lambda_{W}))\subseteq X_{\sg-1}(C)$ clearly.
Therefore we have   
$\sg-2=T^{X_{t(\lambda_{W})-1}(C)}(\Tr_{C}M)=T^{Y^{C}(\Tr_{C}M)}(\Tr_{C}M)$ by Therorem $3.1$ $(B)$ and Proposition $2.3$ $(5)$.
Since $\lambda M \in \ag$,  we have
$\sg-2=T^{X_{t(\lambda_{W})-1}(C)}(\Tr M)=T^{Y^{C}(\Tr_{C}M)}(\Tr M)$.
Then we get  $\sg-1=T^{X_{t(\lambda_{W})-1}(C)}(\lambda M)=T^{Y^{C}(\Tr_{C}M)}(\lambda M)$ by Proposition $2.2$ $(2)$, $2.3$ $(1)$ 
and $T^{Y^{C}(\Tr_{C}M)}(\Tr M)<T^{Y^{C}(\Tr_{C}M)}(R)$ .
Since $Y^{C}(\Tr_{C}M)\subseteq NG_{C}(M)\subseteq X_{t(\lambda_{W})-1}(C)$,
the assertion follows.
\end{proof}
\begin{corollary}\cite[Theorem 4.12]{S2}
Let $R$ be a local ring, $M$ an $R$-module with $0<G_{C}$-$dim_{R}(M)<\infty$ and $\lambda M \in \ag$.
If $M$ is horizontally linked then the following conditions are equivalent.
\begin{enumerate}[$(i)$]
\item
$\depth (M)=syz(M)=\mathrm{r.grade}(\lambda M)$;
\item
$m\in \Ass_{R}(\Ext_{R}^{\mathrm{r.grade}(\lambda M)}(\lambda M, R))$;
\item
$\depth M \le \depth M_{p}$ for each $p\in NG_{C}(M)$.
\end{enumerate}
\end{corollary}
\begin{proof}
we assume that $W=M$ and $\lambda_{W}; M\rightarrow M^{\dag \dag}$ is natural map. 
 Note that
    $\mathrm{r.grade}(\lambda M)=\mathrm{r.grade}(\Tr M)-1$ and $(ii)$ if and only if $m\in \Ass(\Ext_{R}^{\mathrm{r.grade}(\Tr M)}(\Tr M, R))$ holds because $M$ is horizontally linked.
  We have $\Ext_{R}^{i}(\Tr_{C}M, C)\cong \Ext_{R}^{i}(\Tr M, R)$  for $i\ge 0$ ; see proof of \cite[Theorem $4.6$]{S2}. 
  We have $\sg=\SSS(M)$ by \cite[Proposition $2.4$]{S2} because $G_{C}$-$dim_{R}(M)<\infty$.
   Then the assetion follows by Theorem $3.1$ $(A)$ and Proposition $2.2(1)$ since $NG_{C}(M)=Y^{C}(M)$.  
\end{proof}

% $\diamond$ ACNOWLEDGEMENTS $\diamond$\\
% The auther would like to thank his surpervisor Ryo Takahashi for valuable comments.  

\end{document}